\documentclass[10pt,conference,twocolumn]{ieeeconf}

\IEEEoverridecommandlockouts
\usepackage{makeidx}         
\usepackage{graphicx}        
\usepackage{multicol}        
\usepackage[bottom]{footmisc}
\usepackage{latexsym}

\usepackage[pdftex,bookmarks=true]{hyperref}

\usepackage[caption=false,font=footnotesize]{subfig}
\usepackage{cite,color,comment,xspace}
\usepackage{amsfonts}
\usepackage{amsmath}
\usepackage{amssymb}
\usepackage{algorithm}
\usepackage{algorithmic}
\usepackage{url}
\usepackage{times}
\usepackage{enumerate}
\usepackage{epstopdf}
\usepackage{epsfig}
\usepackage{array}
\usepackage{fixltx2e}

\graphicspath{{../fig/}}

\renewcommand{\baselinestretch}{.985}

\newtheorem{theorem}{Theorem}[section]
\newtheorem{proposition}[theorem]{Proposition}

\newtheorem{definition}[theorem]{Definition}

\newtheorem{problem}[theorem]{Problem}

\newcommand{\ev}{\Diamond}
\newcommand{\gl}{\Box}
\newcommand{\un}{\textrm{ }\mathcal U}

\newcommand{\notltl}{\neg}
\newcommand{\andltl}{\wedge}
\newcommand{\orltl}{\vee}
\newcommand{\nextltl}{\bigcirc}

\newcommand{\be}{\begin{equation}}
\newcommand{\ee}{\end{equation}}
\newcommand{\ben}{\begin{equation*}}
\newcommand{\een}{\end{equation*}}
\newcommand{\bea}{\begin{eqnarray}}
\newcommand{\eea}{\end{eqnarray}}
\newcommand{\bean}{\begin{eqnarray*}}
\newcommand{\eean}{\end{eqnarray*}}
\newcommand{\ba}{\begin{array}}
\newcommand{\ea}{\end{array}}
\newcommand{\leftm}{\left[\begin{array}}
\newcommand{\rightm}{\end{array}\right]}

\newcommand{\ie}{{\it i.e., }}
\newcommand{\eg}{{\it e.g., }}
\def\qed{\hfill\rule[-1pt]{5pt}{5pt}\par\medskip}

\def\qed{\hfill\rule[-1pt]{5pt}{5pt}\par\medskip}

\newcommand\oprocendsymbol{\hbox{$\bullet$}}
\newcommand\oprocend{\relax\ifmmode\else\unskip\hfill\fi\oprocendsymbol}

\makeindex   

\title{LTL Control in Uncertain Environments \\with Probabilistic Satisfaction Guarantees\\ - Technical Report -} 

\author{Xu Chu Ding \qquad Stephen L. Smith \qquad Calin Belta \qquad Daniela Rus\thanks{This work was supported in part by ONR-MURI N00014-09-1051, ARO W911NF-09-1-0088, AFOSR YIP FA9550-09-1-020, and NSF CNS-0834260.}
\thanks{X. C. Ding and C. Belta are with Department of Mechanical Engineering, Boston University, Boston, MA 02215, USA (email: {\{xcding; cbelta\}@bu.edu)}.   S. L. Smith is with the Department of Electrical and Computer Engineering, University of Waterloo, Waterloo ON, N2L 3G1 Canada (email: stephen.smith@uwaterloo.ca).  D. Rus is with the Computer Science and Artificial Intelligence Laboratory, Massachusetts Institute of Technology, Cambridge, MA 02139, USA  (email: rus@csail.mit.edu).}
}         

\begin{document}
\maketitle \thispagestyle{empty} \pagestyle{empty}

\begin{abstract} 
  We present a method to generate a robot control strategy that maximizes the probability to accomplish a task.  The task is given as a Linear Temporal Logic (LTL) formula over a set of properties that can be satisfied at the regions of a partitioned environment.  We assume that the probabilities with which the properties are satisfied at the regions are known, and the robot can determine the truth value of a proposition only at the current region. Motivated by several results on partitioned-based abstractions, we assume that the motion is performed on a graph. To account for noisy sensors and actuators, we assume that a control action enables several transitions with known probabilities.  We show that this problem can be reduced to the problem of generating a control policy for a Markov Decision Process (MDP) such that the probability of satisfying an LTL formula over its states is maximized.  We provide a complete solution for the latter problem that builds on existing results from probabilistic model checking. We include an illustrative case study.  \end{abstract}

\section{Introduction}
\label{sec:intro}

Recently there has been an increased interest in using temporal logics, such as Linear Temporal Logic (LTL) and Computation Tree Logic (CTL) as motion specification languages for robotics  \cite{Hadas-ICRA07,Karaman_mu_09,KB-TAC08-LTLCon,Loizou04,Quottrup04,Tok-Ufuk-Murray-CDC09}.  Temporal logics are appealing because they provide formal, high level languages in which to describe complex missions, {\it e.g.,} ``Reach $A$, then $B$, and then $C$, in this order, infinitely often. Never go to $A$. Don't go to $B$ unless $C$ or $D$ were visited."  In addition, off-the-shelf model checking algorithms \cite{Clarke99,Emerson90} and temporal logic game strategies \cite{Piterman-2006} can be used to verify the correctness of robot trajectories and to synthesize robot control strategies.

Motivated by several results on finite abstractions of control systems, in this paper we assume that the motion of the robot in the environment is modeled as a  finite labeled transition system. This can be obtained by simply partitioning the environment and labeling the edges of the corresponding quotient graph according to the motion capabilities of the robot among the
regions. Alternatively, the partition can be made in the state space of the
robot dynamics, and the transition system is then a finite abstraction of a
continuous or hybrid control system \cite{Alur00,Pappas03}. 

The problem of controlling a finite transition system from a temporal logic specification has received a lot of attention during recent years. All the existing works assume that the current state can be precisely determined.  If the result of a control action is deterministic ({\it i.e.,} at each state, an available control enables exactly one transition), control strategies from specifications given as LTL formulas can be found through a simple adaptation of off-the-shelf model checking algorithms \cite{KB-TAC08-LTLCon}. If the control is nondeterministic (an available control at a state enables one of several transitions, and their probabilities are not known), the control problem from an LTL specification can be mapped to the solution of a Rabin game~\cite{Thomas02}
, or simpler B\"uchi and GR(1) games if the specification is restricted to fragments of LTL~\cite{Hadas-ICRA07}.  If the control is probabilistic (an available control at a state enables one of several transitions, and their probabilities are known), the transition system is a Markov Decision Process (MDP).  The control problem then reduces to generating a policy (adversary) for an MDP such that the produced language satisfies a formula of a probabilistic temporal logic \cite{Alfaro95modelchecking,KNP04b}. We have recently developed a framework for deriving an MDP control strategy from a formula in a fragment of probabilistic CTL (pCTL) \cite{LaWaAnBe-ICRA10}.  For probabilistic LTL, in \cite{baier2004controller}, a control strategy is synthesized for an MDP where some states are under control of the environment, so that an LTL specification is guaranteed to be satisfied under all possible environment behaviors.  The temporal logic control problems for systems with probabilistic or nondeterministic state-observation models, which include the class of Partially Observable Markov Decision Processes \cite{PineauT:AAAI:2002,ZhangZ:JAIR:2001}, are currently open.

In this paper, we consider motion specifications given as arbitrary LTL formulas over a set of properties that can be satisfied with given probabilities at the vertices of a graph environment.  We assume that the truth values of the properties can be observed only when a vertex is reached in the environment, and the observations of these properties are independent with each other.  We assume a probabilistic robot control model and that the robot can determine its current vertex precisely. Under these assumptions, we develop an algorithm to generate a control strategy that maximizes the probability of satisfying the specification. Our approach is based on mapping this problem to the problem of generating a control policy for a MDP such that the probability of satisfying an LTL formula is maximized. We provide a solution to this problem by drawing inspiration from probabilistic model checking.  We illustrate the method by applying it to a numerical example of a robot navigating in an indoor environment.

The contribution of this work is twofold. First, we adapt existing approaches in probabilistic model checking (\eg \cite{baier2008principles,vardi1999probabilistic}), and provide a complete solution to the general problem of controlling MDPs from full LTL specifications using deterministic Rabin automata. This is a significant departure from our previous work on MDP control from pCTL formulas \cite{LaWaAnBe-ICRA10}, since it allows for strictly richer specifications.  The increase in expressivity is particularly important in many robotic applications where the robot is expected to perform some tasks, such as surveillance, repeatedly.  However, it comes at the price of increased computational complexity. Second, we allow for non-determinism not only in the robot motion, but also in the robot's observation of properties in the environment.  This allows us to model a large class of robotic problems in which the satisfaction of properties of interest can be predicted only  probabilistically.  For example, we can model a task where a robot is operating in an indoor environment, and is required to pick-up and deliver items among some rooms.  The robot determines its current location using RFID tags on the floors and walls.  Non-determinism occurs in observations because items may or may not be available when a robot visits a room. Non-determinism also occurs in the motion due to imprecise localization or control actuation.

The remainder of the paper is organized as follows: In Section~\ref{sec:prelim} we introduce the necessary definitions and preliminary results.  In Section~\ref{sec:probformulation} we formulate the problem and describe the technical approach.  In Section~\ref{sec:problemreform} we reformulate this problem onto a MDP and show that two problems are equivalent.  We synthesis our controls strategy in Section~\ref{sec:controlsyn}, and an example of the provided algorithm is shown in Section~\ref{sec:example}.  We conclude in Section~\ref{sec:concl}.
 
\section{Preliminaries}
\label{sec:prelim}
In this section we provide background material on linear temporal logic and Markov decision processes.

\subsection{Linear Temporal Logic}
We employ Linear Temporal Logic (LTL) to describe high level motion specifications. A detailed description of the syntax and semantics of LTL is beyond the scope of this paper and can be found in, for example, \cite{Clarke99}.  Roughly, an LTL formula is built up from a set of atomic propositions $\Pi$, which are properties that can be either true or false, standard Boolean operators $\notltl$ (negation), $\orltl$ (disjunction), $\andltl$ (conjunction), and temporal operators $\nextltl$ (next), $\un$ (until), $\ev$ (eventually), $\gl$ (always) and $\Rightarrow$ (implication). The semantics of LTL formulas are given over words, which is defined as an infinite sequence $o=o_{0}o_{1}\ldots$, where $o_{i} \in 2^{\Pi}$ for all $i$.

We say $o\vDash\phi$ if the word $o$ satisfies the LTL formula $\phi$.  The semantics of LTL is defined recursively.  If $\phi=\pi$ is an LTL formula, where $\pi\in\Pi$, then $\phi$ is true at position $i$ of the word if $\pi\in o_{i}$.  A word satisfies an LTL formula $\phi$ if $\phi$ is true at the first position of the word; $\gl \phi$ means that $\phi$ is true at all positions of the word; $\ev \phi$ means that $\phi$ eventually becomes true in the word;  $\phi_{1}\un\phi_{2}$ means $\phi_{2}$ eventually becomes true and $\phi_{1}$ is true until this happens; $\nextltl\phi$ means that $\phi$ becomes true at next position of the word.  More expressivity can be achieved by combining the above temporal and Boolean operators (several examples are given later in the paper).  An LTL formula can be represented by a deterministic Rabin automaton, which is defined as follows.


\begin{definition}[Deterministic Rabin Automaton]
\label{def:DRA}
A deterministic Rabin automaton (DRA) is a tuple $\mathcal R=(Q,\Sigma,\delta,q_{0},F)$, where (i) $Q$ is a finite set of states; (ii) $\Sigma$ is a set of inputs (alphabet); (iii) $\delta:Q\times\Sigma\rightarrow Q$ is the transition function; (iv) $q_{0}\in Q$ is the initial state; and (v) $F=\{(L_{1},K_{1}),\dots,(L_{k},K_{k})\}$ is a set of pairs where $L_{i},K_{i}\subseteq Q$ for all $i\in\{1,\dots,k\}$.
\end{definition}

A run of a Rabin automaton $\mathcal R$, denoted by $r_{\mathcal R}=q_{0}q_{1}\ldots$, is an infinite sequence of states in $\mathcal R$ such that for each $i\geq 0$, $q_{i+1}\in\delta(q_{i},\alpha)$ for some $\alpha \in \Sigma$.   A run $r_{\mathcal R}$ is {\it accepting} if there exists a pair $(L,K)\in F$ such that 1) there exists $n\geq 0$, such that for all $m\geq n$, we have $q_{m}\notin L$, and 2) there exist infinitely many indices $k$ where $q_{k}\in K$.   
This acceptance conditions means that $r_{\mathcal R}$ is accepting if for a pair $(L,K)\in F$, $r_{\mathcal R}$ intersects with $L$ finitely many times and $K$ infinitely many times.

For any LTL formula $\phi$ over $\Pi$, one can construct a DRA with input alphabet $\Sigma= 2^{\Pi}$ accepting all and only words over $\Pi$ that satisfy $\phi$ (see \cite{gradel2002automata}).  We refer readers to \cite{klein2006experiments
} and references therein for algorithms and to freely available implementations, such as \cite{ltl2dstar}, to translate a LTL formula over $\Pi$ to a corresponding DRA.


\subsection{Markov Decision Process and probability measure}
\label{subsec:MDPandprobmeasure}

We now introduce a labeled Markov decision process, and the probability measure we will use in the upcoming sections.

\begin{definition}[Labeled Markov Decision Process]
\label{def:MDP}
A labeled Markov decision process (MDP) is a tuple $\mathcal M=(\mathcal S, \mathcal U, \mathcal A, \mathcal P, \iota, \Pi, h)$, where (i) $\mathcal S$ is a finite set of states; (ii) $\mathcal U$ is a finite set of actions; (iii) $\mathcal A:\mathcal S\to 2^{\mathcal U}$ represents the set of actions enabled at state $s\in \mathcal S$; (iv) $\mathcal P: \mathcal S\times \mathcal U\times \mathcal S\rightarrow [0,1]$ is the transition probability function such that for all states $s\in \mathcal S$, $\sum_{s'\in \mathcal S}\mathcal P(s,u,s') = 1$ if $u\in \mathcal A(s) \subseteq \mathcal U$ and $\mathcal P(s,u,s') = 0$ if $u\notin \mathcal A(s)$; 
(v) $\iota:\mathcal S\to [0,1]$ is the initial state distribution satisfying $\sum_{s\in \mathcal S}\iota(s)=1$; (vi) $\Pi$ is a set of atomic propositions; and (vii) $h: \mathcal S\rightarrow 2^{\Pi}$ is a labeling function.
\end{definition}

The quantity $\mathcal P(s,u,s')$ represents the probability of reaching the state $s'$ from $s$ taking the control $u\in \mathcal A(s)$.  

We will now define a probability measure over paths in the MDP.  To do this, we define an action function as a function $\mu:\mathcal S\rightarrow \mathcal U$ such that $\mu(s)\in \mathcal A(s)$ for all $s\in \mathcal S$.  An infinite sequence of action functions $M=\{\mu_{0},\mu_{1},\ldots\}$ is called a policy.  One can use a policy to resolve all nondeterministic choices in an MDP by applying the action $\mu_{k}(s_{k})$ at each time-step $k$.  Given an initial state $s_{0}$ such that $\iota(s_{0})>0$, an infinite sequence $r^{M}_{\mathcal M}=s_{0}s_{1}\ldots$ on $\mathcal M$ generated under a policy $M=\{\mu_{0},\mu_{1},\ldots\}$ is called a path on $\mathcal M$ if $\mathcal P(s_{i},\mu_{i}(s_{i}),s_{i+1})>0$ for all $i$. The subsequence $s_{0}s_{1}\ldots s_{n}$ is called a finite path.  If $\mu_{i}=\mu$ for all $i$, then we call this policy a stationary policy.

We define $\mathrm{Paths}^{M}_{\mathcal M}$ and $\mathrm{FPaths}^{M}_{\mathcal M}$ as the set of all infinite and finite paths of $\mathcal M$ under a policy $M$ starting from any state $s_{0}$ where $\iota(s_{0})>0$.  
We can then define a probability measure over the set $\mathrm{Paths}^{M}_{\mathcal M}$ of paths.  For a path $r_{\mathcal M}^{M}=s_{0}s_{1}\ldots s_{n}s_{n+1}\ldots \in\mathrm{Paths}^{M}_{\mathcal M}$, the \emph{prefix} of length $n$ of $r_{\mathcal M}^{M}$ is the finite subsequence $s_{0}s_{1}\ldots s_{n}$.  Let $\mathrm{Paths}^{M}_{\mathcal M}(s_{0}s_{1}\ldots s_{n})$ denote the set of all paths in $\mathrm{Paths}^{M}_{\mathcal M}$ with the prefix $s_{0}s_{1}\ldots s_{n}$.  (Note that $s_{0}s_{1}\ldots s_{n}$ is a finite path in $\mathrm{FPaths}^{M}_{\mathcal M}$.)

Then, the probability measure $\textrm{Pr}^{M}$ on the smallest $\sigma$-algebra over $\mathrm{Paths}^{M}_{\mathcal M}$ containing $\mathrm{Paths}^{M}_{\mathcal M}(s_{0}s_{1}\ldots s_{n})$ for all $s_{0}s_{1}\ldots s_{n}\in \mathrm{FPaths}^{M}_{\mathcal M}$ is the unique measure satisfying

\bea
&&\textrm{Pr}^{M}\{\mathrm{Paths}^{M}_{\mathcal M}(s_{0}s_{1}\ldots s_{n})\}\nonumber\\&=&\iota(s_{0})\prod_{0\leq i < n}\mathcal P(s_{i},\mu_{i}(s_{i}),s_{i+1}).
\eea

Finally, we can define the probability that a policy $M$ in an MDP $\mathcal M$ satisfies an LTL formula $\phi$.  
A path $r^{M}_{\mathcal M}=s_{0}s_{1}\ldots$ deterministically generates a word $o=o_{0}o_{1}\ldots$ where $o_{i}=h(s_{i})$ for all $i$. With a slight abuse of notation, we denote $h(r^{M}_{\mathcal M})$ as the word generated by $r^{M}_{\mathcal M}$.  Given an LTL formula $\phi$, one can show that the set $\{r^{M}_{\mathcal M} \in \mathrm{Paths}^{M}_{\mathcal M} : h(r^{M}_{\mathcal M}) \vDash \phi\}$ is measurable.  We define 
\be
\label{eq:probofformula}
\textrm{Pr}^{M}_{\mathcal M}(\phi):=\textrm{Pr}^{M}\{r^{M}_{\mathcal M} \in \mathrm{Paths}^{M}_{\mathcal M} : h(r^{M}_{\mathcal M}) \vDash \phi\} \ee as the probability of satisfying $\phi$ for $\mathcal M$ under policy $M$.  For more details about probability measures on MDPs under a policy and measurability of LTL formulas, we refer the reader to a text in probabilistic model checking, such as \cite{baier2008principles}.  

\section{Model, Problem Formulation, and Approach}
\label{sec:probformulation}

In this section we formalize the environment model, the robot motion model, and the robot observation model.  We then formally state our problem and provide a summary of our technical approach.

\subsection{Environment, task, and robot model}

\subsubsection{Environment model} In this paper, we consider a robot moving in a partitioned environment, which can be represented by a graph and a set of properties:
\be
\label{env}
\mathcal E=(V,\delta_{\mathcal E}, \Pi), \ee where $V$ is the set of vertices, $\delta_{\mathcal E}\subseteq V\times V$ is the relation modeling the set of edges, and $\Pi$ is the set of properties (or atomic propositions).  Such a finite representation of the environment can be obtained by using popular partition schemes, such as triangulations or rectangular grids.  The set $V$ can be considered as a set of labels for the regions in the partitioned environment, and $\delta_{\mathcal E}$ is the corresponding adjacency relation.  
In this paper we assume that there is no blocking vertex in $V$ (\ie all vertices have at least one outgoing edge).

\subsubsection{Task specification}
The atomic propositions $\Pi$ represent properties in the environment that can be true of false.  We require the motion of the robot in the environment to satisfy a rich specification given as an LTL formula $\phi$ over $\Pi$ (see Sec.~\ref{sec:prelim}).  A variety of robotic tasks can be easily translated to LTL formulas.  For example,
\begin{itemize}
\item Parking: ``Find parking lot and then park'' \\($(\ev \textrm{parking lot}) \andltl (\textrm{parking lot} \Rightarrow \nextltl \textrm{park})$)
\item Data Collection: ``Always gather data at gathering locations and then upload the data, repeat infinitely many times'' ($\gl \ev (\textrm{gather} \Rightarrow \ev \textrm{upload})$)
\item Ensure Safety: ``Achieve task $\psi$ while always avoiding states satisfying $P_{1}$ or $P_{2}$'' ($\gl \notltl(P_{1}\orltl P_{2}) \andltl \psi $).
\end{itemize}
\subsubsection{Robot motion model} The motion capability of the robot in the environment is represented by a set of motion primitives $U$, and a function $A:V\rightarrow 2^{U}$ that returns the set of motion primitives available (or enabled) at a vertex $v\in V$.  For example, $U$ can be $\{\textrm{Turn Left}, \textrm{Turn Right},\textrm{Go Straight}\}$ in an urban environment with roads and intersections.  To model non-determinism due to possible actuation or measurement errors, we define the transition probability function $P_{m}:V\times U\times V\rightarrow[0,1]$ such that $\sum_{v'\in V}P_{m}(v,u,v')=1$ for all $v\in V$ and $u\in A(v)$, and $P_{m}(v,u,v')=0$ if $(v,v')\notin\delta_{\mathcal E}$ or if $u\notin A(v)$.  
Thus, $P_{m}(v,u,v')$ is the probability that after applying the motion primitive $u$ at vertex $v$, the robot moves from $v$ to an adjacent region $v'$ without passing through other regions. 
The set $U$ corresponds to a set of feedback controllers for the robot. 
Such feedback controllers can be constructed from facet reachability (see \cite{HS04, Belta-TRO05}), and the transition probabilities can be obtained from experiments (see \cite{LaWaAnBe-ICRA10}).  Note that this model of motion uses an underlying assumption that transition probabilities of the robot controllers do not depend on the previous history of the robot.  


\subsubsection{Robot observation model} In our earlier work \cite{KB-TAC08-LTLCon}, we assumed that the motion of the robot in the partitioned environment is deterministic, and we proposed an automatic framework to produce a provably correct control strategy so that the trajectory of the robot satisfies an LTL formula.  In \cite{LaWaAnBe-ICRA10}, we relaxed this restriction and allowed non-determinism in the motion of the robot, and a control strategy for the robot was obtained to maximize the probability of satisfying a task specified by a fragment of CTL.  In both of these results, it was assumed that some propositions in $\Pi$ are associated with each region in the environment (\ie for each $v\in V$), and they are fixed in time.   

However, this assumption is restrictive and often not true in practice.  For example, the robot might move to a road and find it congested; while finding parking spots, some parking spots may already be taken; or while attempting to upload data at an upload station, the upload station might be occupied.    We wish to design control strategies that react to information which is observed in real-time, \eg if a road is blocked, then pick another route. 


Motivated by these scenarios, in this paper we consider the problem setting where observations of the properties of the environment are probabilistic.  To this end, we define a probability function $P_{o}: V\times\Pi\rightarrow [0,1]$.  Thus, $P_{o}(v,\pi)$ is the probability that the atomic proposition $\pi\in\Pi$ is observed at a vertex $v\in V$ when $v$ is visited.  We assume that all observations of atomic propositions for a vertex $v\in V$ are independent and identically distributed. This is a reasonable model in situations where the time-scale of robot travel is larger than the time scale on which the proposition changes.  For future work, we are pursuing more general observation models. 
Let $\Pi_{v}:=\{\pi\in\Pi :P_{o}(v,\pi)>0\}$ be the atomic propositions that can be observed at a vertex $v$. Then $Z_{v}=\{Z\in 2^{\Pi_{v}} : \underset{\pi\in Z}{\prod}P_{o}(v,\pi)\times \underset{\pi\notin Z}{\prod} (1-P_{o}(v,\pi))>0\}$ is the set of all possible observations at $v$.   

\subsection{Problem Formulation}

Let the initial state of the robot be given as $v_{0}$. The trajectory of the robot in the environment is an infinite sequence $r=v_{0}v_{1},\ldots $, where $P_{m}(v_{i},u,v_{i+1})>0$ for some $u$ for all $i$. Given $r=v_{0}v_{1},\ldots $, we call $v_{i}$ the state of the robot at the discrete time-step $i$.  We denote the observed atomic propositions at time-step $i$ as $o_{i}\in Z_{v_{i}}$ and $O(r)=o_{0}o_{1}\ldots$ as the word observed by $r$.  An example of a trajectory $r$ and its observed word in an environment with given $\mathcal E$, $U$, $A$, $P_{m}$ and $P_{o}$ are shown in Fig.~\ref{fig:Sys}.



\begin{figure}[h]
\begin{center}
\includegraphics[scale=0.4]{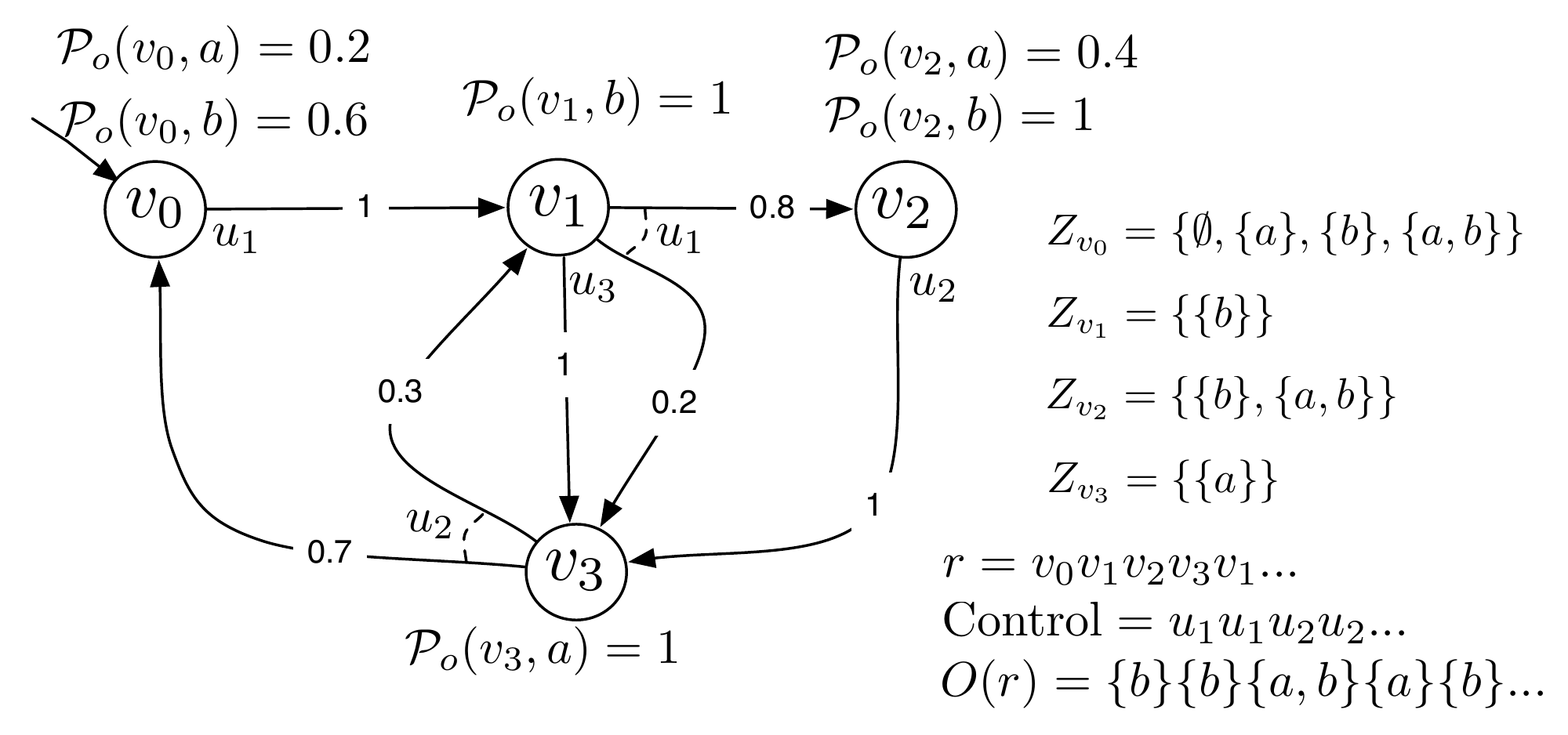}
\caption{An example trajectory $r$ and its observed word $O(r)$
  .  We also show $Z_{v}$ for all $v\in V$.  A single arrow pointed towards a state $v_{0}$ indicates the initial state. The atomic proposition set is $\Pi=\{a,b\}$.  The set of motion primitives is $U=\{u_{1},u_{2},u_{3}\}$.  The probability function $P_{o}$ assigns probabilities for all atomic propositions at each state.  We show the probability of an atomic proposition only if it is positive (\ie $\pi\in\Pi_{v}$).  The number on top of an arrow pointing from a vertex $v$ to $v'$ is the probability $P_{m}(v,u,v')$ associated with a control $u\in U$.}
\label{fig:Sys}
\end{center}
\end{figure}

Our desired ``reactive'' control strategy is in the form of an infinite sequence $C=\{\mathcal \nu_{0},\nu_{1},\ldots\}$ where $\nu_{i}:V\times2^{\Pi}\rightarrow U$ and $\nu_{i}(v, Z)$ is defined only if $Z\in Z_{v}$.  Furthermore, we enforce that $\nu_{i}(v,Z)\in A(v)$ for all $v$ and all $i$.  The reactive control strategy returns the control to be applied at each time-step, given the current state $v$ and observed set of propositions $Z$ at $v$.  Given an initial condition $v_{0}$ and a control strategy $C$, we can produce a trajectory $r=v_{0}v_{1}\ldots $ where the control applied at time $i$ is $\nu_{i}(v_{i},o_{i})$.  We call $r$ and $O(r)=o=o_{0}o_{1}\ldots $ the trajectory and the word generated under $C$, respectively.  Note that given $v_{0}$ and a control strategy $C$, the resultant trajectory and its corresponding word are not unique due to non-determinism in both motion and observation of the robot.

Now we formulate the following problem:
\begin{problem}
\label{prob:maxprob}
Given the environment represented by $\mathcal E=(V,\delta_{\mathcal E},\Pi)$; the robot motion model $U$, $A$ and $P_{m}$; the observation model $P_{o}$; and an LTL formula $\phi$ over $\Pi$, find the control strategy $C$ that maximizes the probability that the word generated under $C$ satisfies $\phi$.
\end{problem}


\subsection{Summary of technical approach}

Our approach to solve Prob.~\ref{prob:maxprob} proceeds by construction of a labeled MDP $\mathcal M$ (see Def.~\ref{def:MDP}), which captures all possible words that can be observed by the robot.  Furthermore, each control strategy $C$ corresponds uniquely to a policy $M$ on $\mathcal M$.   Thus, each trajectory with an observed word under a control strategy $C$ corresponds uniquely to a path on $\mathcal M$ under $M$.  We then reformulate Prob.~\ref{prob:maxprob} as the problem of finding the policy on $\mathcal M$ that maximizes the probability of satisfying $\phi$. These two problems are equivalent due to the assumption that all observations are independent.   We synthesize the optimal control strategy by solving maximal reachability probability problems inspired by results in probabilistic model checking.  
Our framework is more general than in \cite{LaWaAnBe-ICRA10} due to a richer specification language and non-determinism in observation of the environment.  The trade off is that computational complexity in this approach is in general much larger due to increased size of the automaton representing the specification.  

\section{MDP Construction and Problem Reformulation}
\label{sec:problemreform}
As part of our approach to solve Problem~\ref{prob:maxprob}, we construct a labeled MDP $\mathcal M=(\mathcal S, \mathcal U, \mathcal A, \mathcal P, \iota, \Pi, h)$ from the environment model $\mathcal E$, the robot motion model $U$, $A$, $P_{m}$, and the observation model $P_{o}$ as follows:
\begin{itemize}
\item $\mathcal S=\{(v,Z) \,|\, v\in V, Z\in Z_{v}\}$
\item $\mathcal U=U$
\item $\mathcal A((v,Z))=A(v)$
\item $\mathcal P((v,Z),u,(v',Z'))=$
\ben  P_{m}(v,u,v')\times\left(\underset{\pi\in Z'}{\prod}P_{o}(v',\pi)\times \underset{\pi\notin Z'}{\prod} (1-P_{o}(v',\pi))\right)
\een
\item $\iota$ is defined as 
$\iota(s)=\underset{\pi\in Z}{\prod}\mathcal P(v_{0},\pi)\times \underset{\pi\notin Z}{\prod} (1-\mathcal P(v_{0},\pi))$ if $s=(v_{0},Z)$ for any $Z\in Z_{v_{0}}$, and $\iota(s)=0$ otherwise.
\item $h((v,Z))=Z$ for all $(v,Z)\in S$.
\end{itemize}

An example of a constructed MDP is shown in Fig.~\ref{fig:MDPexp}.  One can easily verify that $\mathcal M$ is a valid MDP such that for all $s\in \mathcal S$, $\sum_{s'\in \mathcal S}\mathcal P(s,u,s') = 1$ if $u\in \mathcal A(s)$, $\mathcal P(s,u,s') = 0$ if $u\notin \mathcal A(s)$, and $\sum_{s\in \mathcal S}\iota(s)=1$.  We discuss the growth of the state space from $\mathcal E$ to $\mathcal M$ in Section~\ref{sec:complexity}.
\begin{figure}[h]
\begin{center}
\includegraphics[scale=0.44]{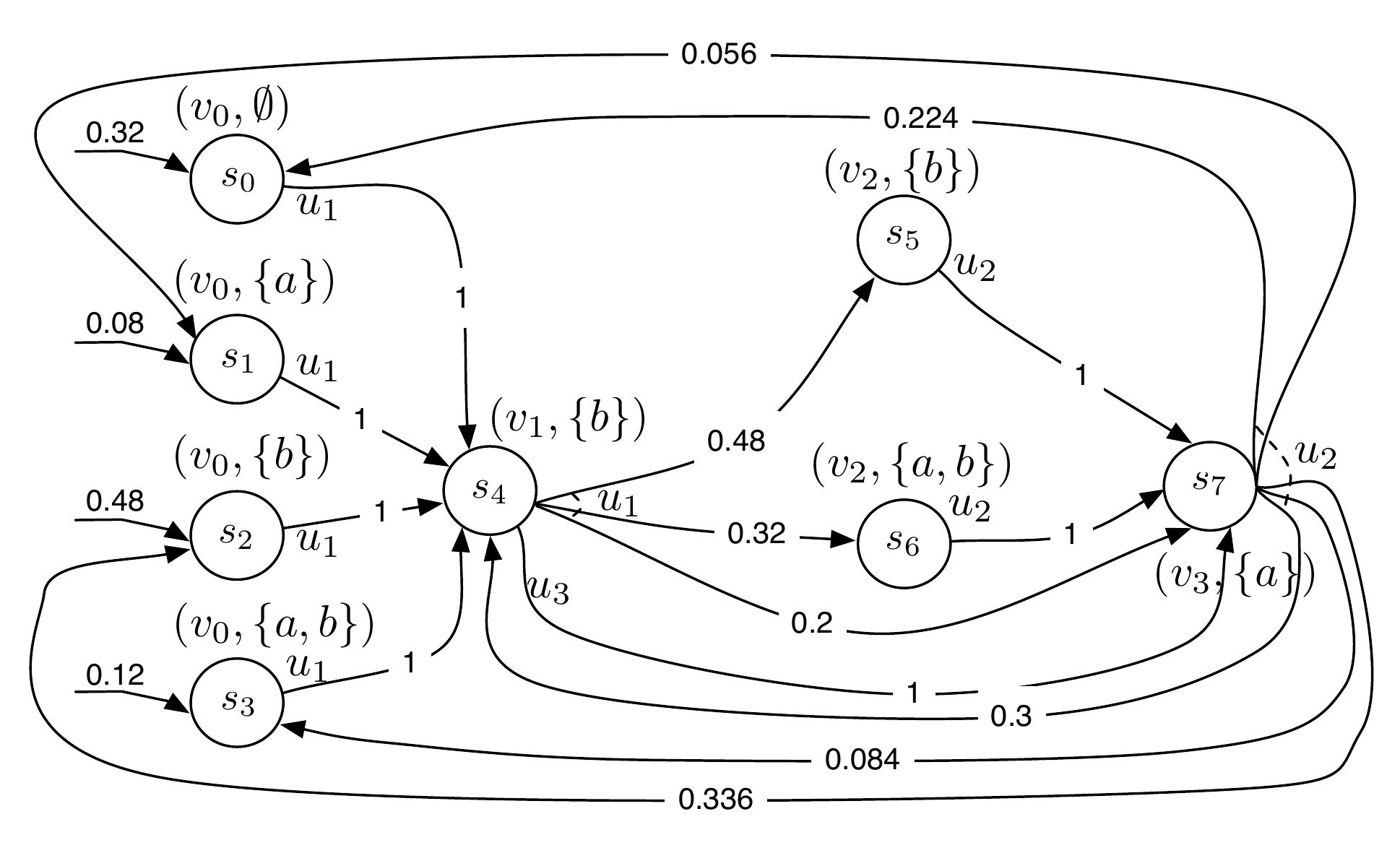}
\caption{The constructed MDP $\mathcal M$ using $\mathcal E$, $U$, $A$, $P_{m}$ and $P_{o}$ from the example in Fig.~\ref{fig:Sys}.  For each state $s\in \mathcal S$, the labels on top of the state show the components of $s$ (\ie $s=(v,Z)$).  The number on the arrow from the state $(v,Z)$ to the state $(v',Z')$ denotes the transition probability $\mathcal P((v,Z),u,(v',Z'))$ for the action $u\in\mathcal U$.  The numbers atop arrows pointing into states $(v_{0},Z_{v_{0}})$ denote the initial distribution.  The set of atomic propositions assigned to each state in $\mathcal M$ is the second component of the state.}
\label{fig:MDPexp}
\end{center}
\end{figure}

 
We now formulate a problem on the MDP $\mathcal M$.  We will then show that this new problem is equivalent to  Prob.~\ref{prob:maxprob}.
\begin{problem}
\label{prob:reformmaxprob}
For a given labeled MDP $\mathcal M$ and an LTL formula $\phi$, find a policy such that $\textrm{Pr}^{M}_{\mathcal M}(\phi)$ (see Eq.~(\ref{eq:probofformula})) is maximized.
\end{problem}

The following proposition formalizes the equivalence between the two problems, and the one-to-one correspondence between a control strategy on $\mathcal E$ and a policy on~$\mathcal M$.

\begin{proposition}[Equivalence of problems]
\label{prop:equivalence}
A control strategy $C=\{\nu_{0},\nu_{1},\ldots\}$ is a solution to Problem~\ref{prob:maxprob} if and only if
the policy $M=\{\mu_{0},\mu_{1},\ldots\}$, where
\[
\mu_{i}\big((v_{i},Z_{i})\big)=\nu_{i}(v_{i},Z_{i}) \quad \text{for each $i$},
\] 
is a solution to Problem~\ref{prob:reformmaxprob}.
\end{proposition}
\begin{proof}
We can establish an one-to-one correspondence between a control strategy $C$ in the environment $\mathcal E$ and a policy $M$ on $\mathcal M$.  Given $C=\{\nu_{0},\nu_{1},\ldots\}$, we can obtain the corresponding $M=\{\mu_{0},\mu_{1},\ldots\}$ by setting $\mu_{i}((v_{i},Z_{i}))=\nu_{i}(v_{i},Z_{i})$.  Conversly, given $M=\{\mu_{0},\mu_{1},\ldots\}$, we can generate a corresponding control strategy $C=\{\nu_{0},\nu_{1},\ldots\}$ such that $\nu_{i}(v_{i},Z_{i})=\mu_{i}((v_{i},Z_{i}))$.

We need only to verify that we can use the same probability measure on paths in $\mathcal E$ and on trajectories in $\mathcal M$.
Due to the assumption that each observation at $v$ is independent, the observation process is Markovian as it only depends on which vertex the observation is made.  Note that the probability of observing $Z\in Z_{v}$ at a state $v\in V$ is $\prod_{\pi\in Z}\mathcal P(v,\pi)\times \prod_{\pi\notin Z} (1-\mathcal P(v,\pi))$.   Hence, the probability of moving from a vertex $v$ to $v'$, under control $u$, and observing $Z'\in Z_{v'}$ is exactly $\mathcal P((v,Z),u(v',Z'))$.   Therefore, the probability of observing the finite word $O(r^{f})$ for a finite trajectory $r^{f}=v_{0}\ldots v_{n}$ under $C$ is the same (by construction of $\mathcal M$) as traversing through a finite path $fr^{M}_{\mathcal M}\in \mathrm{FPaths}^{M}_{\mathcal M}$ such that $h(fr^{M}_{\mathcal M})=O(r^{f})$ under the policy $M$ corresponding to $C$.  Since this property holds for any arbitrary finite trajectory $r^{f}$, a trajectory $r$ with a word $O(r)$ under $C$ can be uniquely mapped to a path in $\mathrm{Paths}^{M}_{\mathcal M}$ and we can use the probability measure and $\sigma$-algebra (see Sec.~\ref{subsec:MDPandprobmeasure}) on $\mathcal M$ under a policy $M$ for the corresponding control strategy $C$.  Thus, if $M$ is a solution for Prob.~\ref{prob:maxprob}, then the control strategy $C$ corresponding to $M$ is a solution for Prob.~\ref{prob:reformmaxprob}, and vice versa. \qed 
\end{proof}
Due to the above proposition, we will proceed by constructing a policy $M$ on the MDP $\mathcal M$ as a solution to Prob.~\ref{prob:reformmaxprob}.  We can then uniquely map $M$ to a control strategy $C$ in the robot environment $\mathcal E$ for a solution to Prob.~\ref{prob:maxprob}.


\section{Synthesis of Control Strategy}
\label{sec:controlsyn}
In this section we provides a solution for Prob.~\ref{prob:maxprob} by synthesizing an optimal policy for Prob.~\ref{prob:reformmaxprob}.  Our approach is adapted from automata-theoretic approaches in the area of probabilistic verification and model checking (see \cite{vardi1999probabilistic} and references therein for an overview).   Probabilistic LTL model checking finds the maximum probability that a path of a given MDP satisfies a LTL specification.  We modify this method to obtain an optimal policy that achieves the maximum probability.  This approach is related to the work of \cite{courcoubetis1998markov}, in which rewards are assigned to specifications and non-deterministic B\"{u}chi automata (NBA) are used.  We do not use NBAs since a desired product MDP cannot be directly constructed from an NBA, but only from an DRA.

\subsection{The Product MDP}
We start by converting the LTL formula $\phi$ to a DRA defined in Def.~\ref{def:DRA}.  We denote the resulting DRA as $\mathcal R_{\phi}=(Q,2^{\Pi},\delta,q_{0},F)$ with $F=\{(L_{1},K_{1}),\ldots ,(L_{k},K_{k})\}$ where $L_{i},K_{i}\subseteq Q$ for all $i=1,\ldots,k$.   The DRA obtained from the LTL formula $\phi=\gl\ev a \andltl \gl\ev b$ is shown in Fig.~\ref{fig:DRA}.

\begin{figure}[h]
\begin{center}
\includegraphics[scale=0.43]{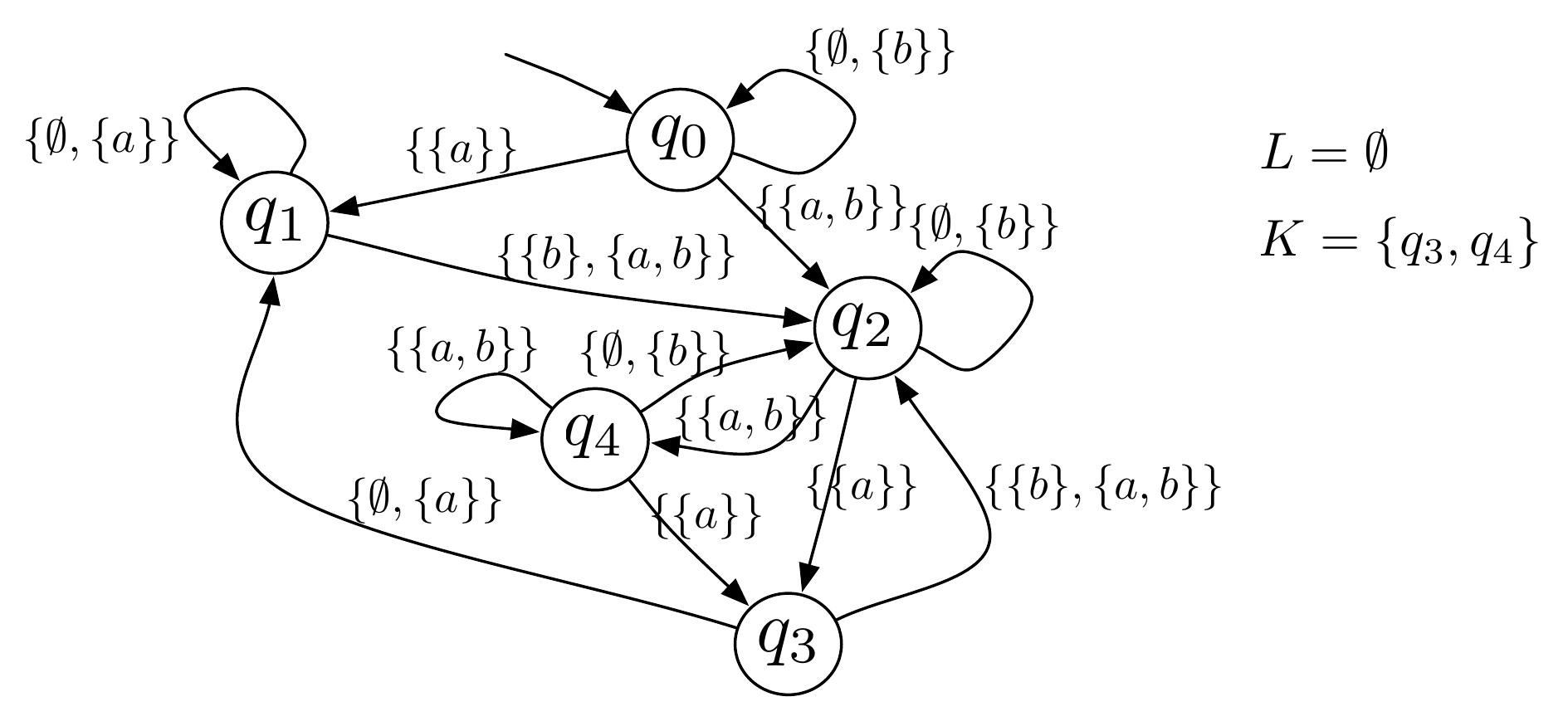}
\caption{The DRA $\mathcal R_{\phi}$ corresponding to the LTL formula $\phi=\gl\ev a \andltl \gl\ev b$.  In this example, there is one set of accepting states $F=\{(L,K)\}$ where $L=\emptyset$ and $K=\{q_{3},q_{4}\}$.   Thus, accepting runs of this DRA must visit $q_{3}$ or $q_{4}$ (or both) infinitely often. }%
\label{fig:DRA}
\end{center}
\end{figure}

We now obtain an MDP as the product of a labeled MDP $\mathcal M$ and a DRA $\mathcal R_{\phi}$.  This product MDP allows one to find runs on $\mathcal M$ that generate words satisfying the acceptance condition of $\mathcal R_{\phi}$.

\begin{definition}[Product MDP]
The product MDP $\mathcal M \times \mathcal R_{\phi}$ between a labeled MDP $\mathcal M=(\mathcal S, \mathcal U, \mathcal A, \mathcal P, \iota, \Pi, h)$ and a DRA $\mathcal R_{\phi}=(Q,2^{\Pi},\delta,q_{0},F)$ is a MDP $\mathcal M_{\mathcal P}=(\mathcal S_{\mathcal P},\mathcal U,\mathcal A_{\mathcal P}, \mathcal P_{\mathcal P},\iota_{\mathcal P})$, where:
\begin{itemize}
\item $\mathcal S_{\mathcal P}= \mathcal S\times Q$ (the Cartesian product of sets $\mathcal S$ and $Q$)
\item $\mathcal A_{\mathcal P}((s,q))=\mathcal A(s)$
\item $\mathcal P_{\mathcal P}((s,q),u,(s',q'))=$
\ben
\left\{\begin{array}{ll} \mathcal P(s,u,s') & \textrm{ if } q'=\delta(q,h(s')) \\ 0 & \rm otherwise \end{array} \right.
\een
\item $\iota_{\mathcal P}((s,q))=\iota(s)$ if $q=\delta(q_{0},h(s))$ and $\iota_{\mathcal P}=0$ otherwise.
\end{itemize}
\end{definition}
We generate the accepting state pairs $F_{\mathcal P}$ for the product MDP $\mathcal M_{\mathcal P}$ as follows: 
For a pair $(L_{i},K_{i})\in F$, a state $(s,q)$ of $\mathcal M_{\mathcal P}$ is in $L^{\mathcal P}_{i}$ if $q\in L_{i}$, and $(s,q)\in K^{\mathcal P}_{i}$ if $q\in K_{i}$.

As an example, we show in Fig.~\ref{fig:product} some of the states and transitions for the product MDP $\mathcal M_{\mathcal P}=\mathcal M\times\mathcal R_{\phi}$ where $\mathcal M$ is shown in Fig.~\ref{fig:MDPexp} and $\mathcal R_{\phi}$ is shown in Fig.~\ref{fig:DRA}.

\begin{figure}[h]
\begin{center}
\includegraphics[scale=0.5]{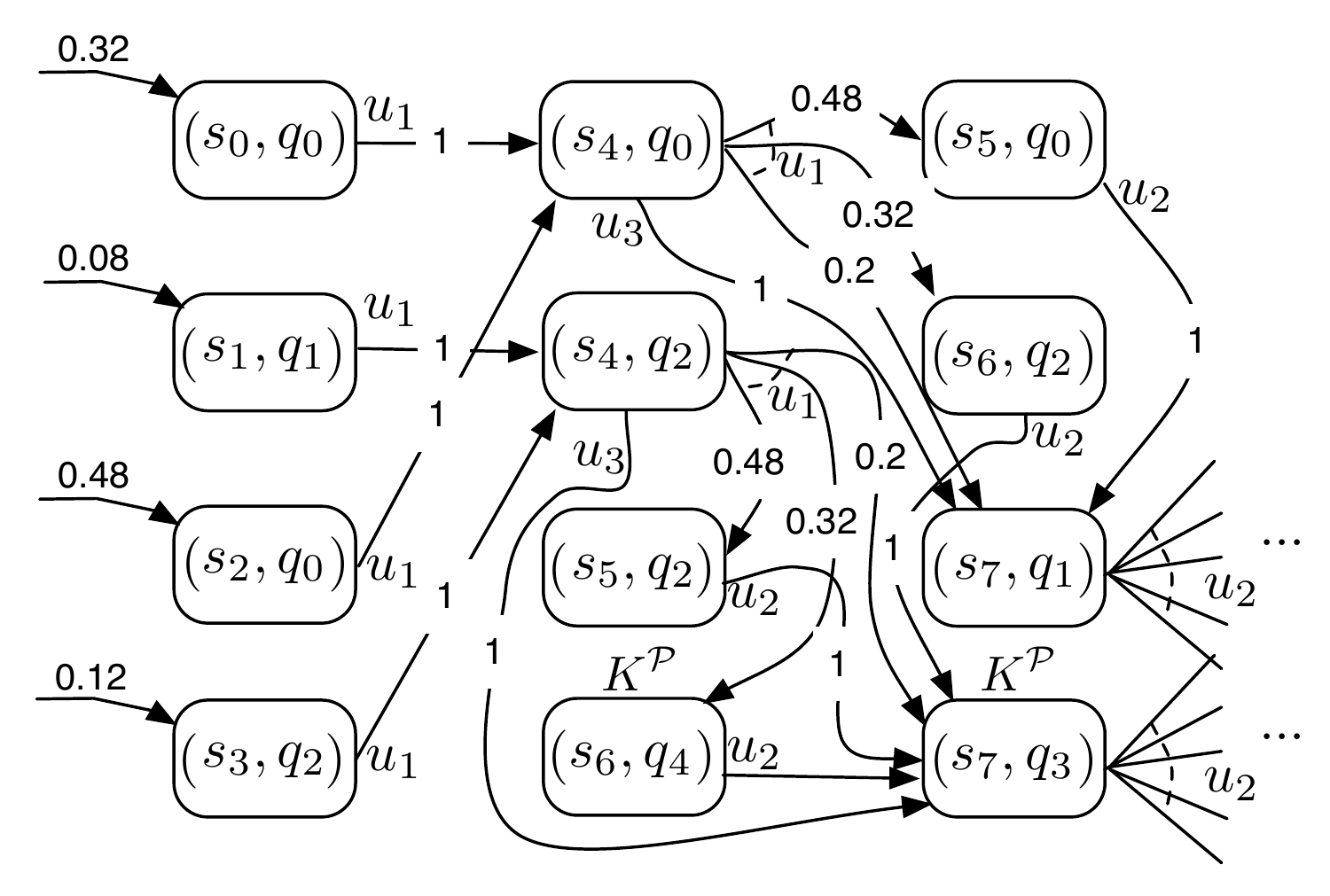}
\caption{The product MDP $\mathcal M_{\mathcal P}=\mathcal M\times\mathcal R_{\phi}$ where $\mathcal M$ is shown in Fig.~\ref{fig:MDPexp} and $\mathcal R_{\phi}$ is shown in Fig.~\ref{fig:DRA}.  Due to space restrictions, only the initial states (states for which the initial distribution $\iota_{\mathcal P}((s,q))$ is positive) and part of the state space are shown.  $F_{\mathcal P}=\{(L^{\mathcal P},K^{\mathcal P})\}$, where $L^{\mathcal P}=\emptyset$ and states in $K^{\mathcal P}$ are marked. }%
\label{fig:product}
\end{center}
\end{figure}

Note that the set of actions for $\mathcal M_{\mathcal P}$ is the same as the one for $\mathcal M$.   A policy $M_{\mathcal P}=\{\mu_{0}^{\mathcal P},\mu_{1}^{\mathcal P},\ldots\}$ on $\mathcal M_{\mathcal P}$ directly induces a policy $M=\{\mu_{0},\mu_{1},\ldots\}$ on $\mathcal M$ by keeping track of the state on the product MDP ($\mu_{i}^{\mathcal P}$ is an action function that returns an action corresponding to a state in $\mathcal M_{\mathcal P}$).  Note that given the state of $\mathcal M$ at time-step $i$ and the state of $\mathcal M_{\mathcal P}$ at time-step $i-1$, the state of $\mathcal M_{\mathcal P}$ at time-step $i$ can be exactly determined.   We can induce a policy $M$ for $\mathcal M$ from a policy $M_{\mathcal P}$ for $\mathcal M_{\mathcal P}$ as follows:
\begin{definition}[Inducing a policy for $\mathcal M$ from $\mathcal M_{\mathcal P}$]
If\\ the state of $\mathcal M_{\mathcal P}$ at time-step $i$ is $(s_{i},q_{i})$, then the policy $M=\{\mu_{0},\mu_{1},\ldots\}$ induced from $M_{\mathcal P}=\{\mu^{\mathcal P}_{0},\mu^{\mathcal P}_{1}\ldots\}$ can be obtained by setting $\mu_{i}(s_{i})=\mu_{i}^{\mathcal P}((s_{i},q_{i}))$ for all $i$.
\end{definition}


We denote $r^{M_{\mathcal P}}_{\mathcal M_{\mathcal P}}$ as a path on $\mathcal M_{\mathcal P}$ under a policy $M_{\mathcal P}$.  We say a path $r^{M_{\mathcal P}}_{\mathcal M_{\mathcal P}}$ is accepting if and only if it satisfies the Rabin acceptance condition with $F_{\mathcal P}$ as the accepting states pairs, \ie there exists a pair $(L^{\mathcal P},K^{\mathcal P})\in F_{\mathcal P}$ so that $r^{M_{\mathcal P}}_{\mathcal M_{\mathcal P}}$ intersects with $L^{\mathcal P}$ finitely many times and $K^{\mathcal P}$ infinitely many times.

The product MDP is constructed so that given a path $r^{M_{\mathcal P}}_{\mathcal M_{\mathcal P}}=(s_{0},q_{0})(s_{1},q_{1})\ldots$, the path $s_{0}s_{1}\ldots$ on $\mathcal M$ generates a word that satisfies $\phi$ if and only if the infinite sequence $q_{0}q_{1}\ldots$ is an accepting run on $\mathcal R_{\phi}$, in which case $r^{M_{\mathcal P}}_{\mathcal M_{\mathcal P}}$ is accepting.   Therefore, each accepting path of $\mathcal M_{\mathcal P}$ uniquely corresponds to a paths of $\mathcal M$ whose word satisfies $\phi$.  

\subsection{Generating the Optimal Control Strategy}
Once we obtain the product MDP $\mathcal M_{\mathcal P}=\mathcal M\times\mathcal R_{\phi}$ and the accepting states pairs $F_{\mathcal P}=\{(L^{\mathcal P}_{1},K^{\mathcal P}_{1}),\ldots,(L^{\mathcal P}_{k},K^{\mathcal P}_{k})\}$, the method to obtain a solution to Prob.~\ref{prob:maxprob} proceeds as follows:  For each pair $(L^{\mathcal P}_{i},K^{\mathcal P}_{i})\in F_{\mathcal P}$, we obtain a set of accepting maximum end components.   An accepting maximum end component for $\mathcal M_{\mathcal P}$ consists of a set of states $\overline{\mathcal S_{\mathcal P}}\subseteq\mathcal S_{\mathcal P}$ and a function $\overline{\mathcal A_{\mathcal P}}$ such that $\emptyset\neq\overline{\mathcal A_{\mathcal P}}((s,q))\subseteq \mathcal A_{\mathcal P}((s,q))$ for all $(s,q)\in\overline{\mathcal S_{\mathcal P}}$.  It has the property that, by taking actions enabled by $\overline{\mathcal A_{\mathcal P}}$, all states in $\overline{\mathcal S_{\mathcal P}}$ can reach every other state in $\overline{\mathcal S_{\mathcal P}}$ and can not reach any state outside of $\overline{\mathcal S_{\mathcal P}}$.  Furthermore, it contains no state in $L^{\mathcal P}_{i}$ and at least one state in $K^{\mathcal P}_{i}$.  In addition, it is called maximum because it is not properly contained in another accepting maximum end component.  Note that for a given pair $(L^{\mathcal P}_{i},K^{\mathcal P}_{i})$, its accepting maximal end components are pairwise disjoint.


A procedure to obtain all accepting maximum end components of an MDP is outlined in \cite{baier2008principles}.  
From probabilistic model checking (see \cite{vardi1999probabilistic, baier2008principles}), the maximum probability of satisfying the LTL formula $\phi$ for $\mathcal M$ is the same as the maximum probability of reaching any accepting maximum end component of $\mathcal M_{\mathcal P}$.   Once an accepting maximum end component $(\overline{\mathcal S_{\mathcal P}},\overline{\mathcal A_{\mathcal P}})$ is reached, then all states in $\overline{\mathcal S_{\mathcal P}}$ are reached infinitely often (and $\phi$ satisfied) with probability 1, under a policy that all actions in $\overline{\mathcal A_{\mathcal P}}$ are used infinitely often.

The maximum probability of reaching a set of states $B_{\mathcal P}\subseteq \mathcal S_{\mathcal P}$ can be obtained by the solution of a linear program.   First we find the set of states that can not reach $B_{\mathcal P}$ under any policy and denote it as $C_{\mathcal P}$ (a simple graph analysis is sufficient to find this set).  We then let $x_{p}$ denote the maximum probability of reaching the set $B_{\mathcal P}$ from a state $p\in\mathcal S_{\mathcal P}$.  We have $x_{p}=1$ if $p\in B_{\mathcal P}$, and $x_{p}=0$ if $p\in  C_{\mathcal P}$.   The values of $x_{p}$ for the remaining states can then be determined by solving a linear optimization problem:
\bea
\label{eq:LP}
&&\min \sum_{p\in S_{\mathcal P}}x_{p}, \textrm{ subject to: }  0\leq x_{p}\leq 1 \textrm{, and } \nonumber\\
&&x_{p}\geq \sum_{t\in S_{\mathcal P}} \mathcal P_{\mathcal P}(p,u,t)x_{t} \textrm{ for all } p\in S_{\mathcal P}\setminus (B_{\mathcal P}\cup C_{\mathcal P}) \nonumber\\&& \textrm{and for all } u\in \mathcal A_{\mathcal P}(p).
\eea

Once $x_{p}$ is obtained for all $p\in S_{\mathcal P}$, one can identify an action $u^{\star}$ (not necessarily unique) for each state $p\in S_{\mathcal P}\setminus (B_{\mathcal P}\cup C_{\mathcal P})$ such that:
\be
\label{eq:optU}
x_{p}=\sum_{t\in S_{\mathcal P}}\mathcal P_{\mathcal P}(p,u^{\star},t)x_{t}.
\ee  
We define a function $\mu^{\star}_{\mathcal P}:\mathcal S_{\mathcal P}\rightarrow \mathcal U$ that returns an action $u^{\star}$ satisfying (\ref{eq:optU}) if $p\in S_{\mathcal P}\setminus (B_{\mathcal P}\cup C_{\mathcal P})$ (the actions for states in $B_{\mathcal P}$ or $C_{\mathcal P}$ are irrelevant and can be chosen arbitrarily).  Then the optimal policy maximizing the probability of reaching  $B_{\mathcal P}$ is the stationary policy $\{\mu^{\star}_{\mathcal P},\mu^{\star}_{\mathcal P},\ldots\}$. 

Our desired policy $M_{\mathcal P}^{\star}$ that maximizes the probability of satisfying $\phi$ on the product MDP is the policy maximizing the probability of reaching the union of all accepting maximum end components of all accepting state pairs in $F_{\mathcal P}$, if the state is not in an accepting maximum end component.  Otherwise, the optimal policy is to use all actions allowed in the associated accepting maximal end component infinitely often in a round-robin fashion.  The solution to Prob.~\ref{prob:reformmaxprob} is then the policy $M^{\star}$ on $\mathcal M$ induced by $M^{\star}_{\mathcal P}$.  The desired reactive control strategy $C^{\star}$ as a solution to Prob.~\ref{prob:maxprob} can finally be obtained as the control strategy corresponding to $M^{\star}$ (see Prop.~\ref{prop:equivalence}).  Our overall approach is summarized in Alg.~\ref{alg:optmaxprob}.


\begin{algorithm}
\caption{Generating the optimal control strategy $C^{\star}$ given $\mathcal E$, $U$, $A$, $P_{m}$, $P_{o}$ and $\phi$}
\begin{algorithmic}[1]
\label{alg:optmaxprob}
\STATE Generate the MDP $\mathcal M$ from the environment model $\mathcal E$, the motion primitives $U$, the actions $A$, the motion model $P_{m}$ and the observation model $P_{o}$.
\STATE Translate the LTL formula $\phi$ to a deterministic Rabin automaton $\mathcal R_{\phi}$.
\STATE Generate the product MDP $\mathcal M_{\mathcal P}=\mathcal M\times \mathcal R_{\phi}$ and accepting states pairs $F_{\mathcal P}=\{(L_{1}^{\mathcal P},K_{1}^{\mathcal P}),\ldots ,(L_{k}^{\mathcal P},K_{k}^{\mathcal P})\}$.
\STATE Find all accepting maximum end components for all pairs $(L_{i}^{\mathcal P},K_{i}^{\mathcal P})\in F_{\mathcal P}$, and find their union $B_{\mathcal P}$.
\STATE Find the stationary policy $\{\mu^{\star}_{\mathcal P}, \mu^{\star}_{\mathcal P},\ldots\}$ maximizing the probability of reaching $B_{\mathcal P}$ by solving (\ref{eq:LP}) and (\ref{eq:optU}).
\STATE Generate the policy $M^{\star}_{\mathcal P}=\{\mu_{0}^{\mathcal P},\mu_{1}^{\mathcal P},\ldots\}$ as follows: $\mu_{i}^{\mathcal P}(p)=\mu^{\star}_{\mathcal P}(p)$ if $p\in \mathcal S_{\mathcal P}\setminus B_{\mathcal P}$.  Otherwise, $p$ is in at least  one accepting maximum end component.  Assuming it is $(\overline{\mathcal S_{\mathcal P}},\overline{\mathcal A_{\mathcal P}})$ and $\overline{\mathcal A_{\mathcal P}}(p)=\{u_{1},u_{2},\ldots,u_{m}\}$, then $\mu_{i}^{\mathcal P}(p)=u_{j}$ where $j=i$ mod $m$.
\STATE Generate the policy $M^{\star}=\{\mu_{0},\mu_{1},\ldots\}$ induced by $M^{\star}_{\mathcal P}$.
\STATE Generate the control strategy $C^{\star}=\{\nu_{0},\nu_{1},\ldots\}$ corresponding to $M^{\star}$ by setting $\nu_{i}(v,Z)=\mu_{i}((v,Z))$ for all~$i$.
\end{algorithmic}
\end{algorithm}

\subsection{Complexity}
\label{sec:complexity}

The complexity of our proposed algorithm is dictated by the size of the generated MDPs.  We use $|\cdot|$ to denote cardinality of a set.  The number of states in $\mathcal M$ is $|\mathcal S|=\sum_{v\in V}|Z_{v}|$.  Hence, in the worst case where all propositions $\pi\in\Pi$ can be observed with positive but less than 1 probability at all vertices $v\in V$, $|\mathcal S|=2^{|\Pi|}$. In practice, the number of propositions that  can be non-deterministically observed at a vertex is small.  For example, in an urban setting, most regions of the environment including intersections and roads have fixed atomic propositions.  The number of intersections that can be blocked is small comparing to the size of the environment. 

The size of the DRA $|Q|$ is in worst case, doubly exponential with respect to $|\Pi|$.  However, empirical studies such as \cite{klein2006experiments} have shown that in practice, the sizes of the DRAs for many LTL formulas are exponential or lower with respect to $|\Pi|$.  In robot control applications, since properties in the environment are typically assigned scarcely (meaning that each region of the environment is usually assigned a small number of properties comparing to $|\Pi|$), the size of DRA can be reduced much further by removing transitions in the DRA with inputs that can never appear in the environment, and then minimizing the DRA by removing states that can not be reached from the initial state.    

The size of the product MDP $\mathcal M_{\mathcal P}$ is $|\mathcal M|\times|Q|$.  The complexity for the algorithm to generate accepting maximal end component is at most quadratic in the size of $\mathcal M_{\mathcal P}$ (see \cite{baier2008principles}), and the complexity for finding the optimal policy from a linear program is polynomial in the size of $\mathcal M_{\mathcal P}$.  Thus, overall, our algorithm is polynomial in the size of $\mathcal M_{\mathcal P}$.


\section{Example}
\label{sec:example}
The computational framework developed in this paper is implemented in MATLAB, and here we provide an example as a case study.  Consider a robot navigating in an indoor environment as shown in Fig.~\ref{fig:room}.  
\begin{figure}[h]
\begin{center}
\includegraphics[scale=0.5]{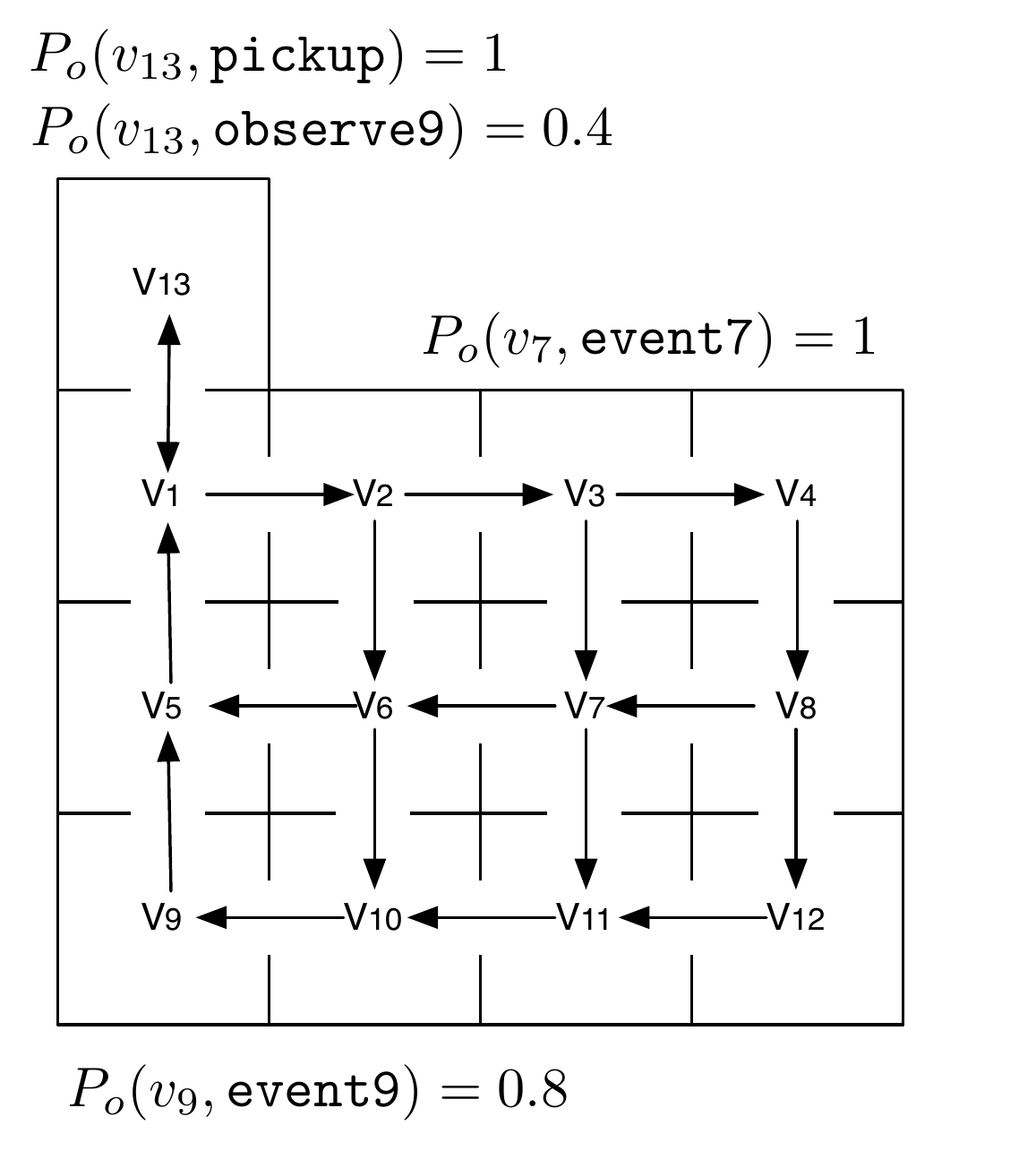}
\caption{Environment for a numerical example of the proposed approach. We assume that the set of motion primitive is $U=\{\alpha, \beta,\gamma\}$.  The number of actions available at each vertex depends on the number of arrows from that vertex to adjacent vertices.  We define the enabling function $A$ so that the motion primitive $\alpha$ is enabled at all vertices, $\beta$ is enabled at vertices $v_{1}$, $v_{6}$ and $v_{7}$, and $\gamma$ is enabled at vertices $v_{2}$, $v_{3}$, $v_{6}$, $v_{7}$ and $v_{8}$.}%
\label{fig:room}
\end{center}
\end{figure}
Each region of the environment is represented by a vertex $v_i$, and the arrows represent allowable transitions between regions.  In this case study, we choose the motion primitives arbitrarily (see the caption of Fig.~\ref{fig:room}).  In practice, they can either correspond to low level control actions such as ``turn left'', ``turn right'' and ``go straight'', or high level commands such as ``go from region 1 to region 2'', which can then be achieved by a sequence of low level control actions.

The goal of the robot is to perform a persistent surveillance mission on regions $v_7$ and $v_9$, described as follows: The robot can pickup (or receive) a surveillance task at region $v_{13}$.  With probability 0.4 the robot receives the task denoted \texttt{observe9}.  Otherwise, the task is \texttt{observe7}.  The task \texttt{observe7} (or \texttt{observe9}) is completed by traveling to region $v_7$ (or $v_9$), and observing some specified event.  In region $v_7$, the robot observes the event ($\texttt{event7}$) with probability $1$.  In region $v_9$, each time the robot enters the region, there is a probability of $0.8$ that it observes the event ($\texttt{event9}$).  Thus, the robot may have to visit $v_9$ multiple times before observing $\texttt{event9}$.  Once the robot observes the required event, it must return to $v_{13}$ and pickup a new task.

This surveillance mission can be represented by four atomic propositions $\{\texttt{pickup},\texttt{observe9},\texttt{event7},\texttt{event9}\}$. (the task $\texttt{observe7}$ can be written as $\notltl \texttt{observe9}$).  The propositions $\texttt{pickup}$ and $\texttt{observe7}$ are assigned to $v_{13}$, with $P_{o}(v_{13},\texttt{pickup}) = 1$ and $P_{o}(v_{13},\texttt{observe9}) = 0.4$.  The proposition $\texttt{event7}$ is assigned to $v_7$ with $P_{o}(v_{7},\texttt{event7}) = 1$ and $\texttt{event9}$ is assigned to $v_9$ with $P_{o}(v_{9},\texttt{event9}) = 0.8$.

The surveillance mission can be written as the following LTL formula:
\begin{align*}
&\phi = \gl\ev \texttt{pickup} \andltl  \nonumber \\ 
& \gl \left(\texttt{pickup}\andltl \notltl \texttt{observe9} \Rightarrow \nextltl (\notltl \texttt{pickup} \un \texttt{event7})\right)\nonumber\\
& \andltl \gl \left(\texttt{pickup}\andltl \texttt{observe9} \Rightarrow \nextltl (\notltl \texttt{pickup} \un \texttt{event9})\right).
\end{align*}

The first line of $\phi$, $\gl\ev \texttt{pickup}$, enforces that the robot must repeatedly pick up tasks.  The second line pertains to task $\texttt{observe7}$ and third line pertains to task $\texttt{observe9}$.  These two lines ensure that a new task cannot be picked up until the current task is completed (\ie the desired event is observed).  Note that if \texttt{event9} is observed after observing \texttt{event7}, then the formula $\phi$ is not violated (and similarly if \texttt{event7} is observed after observing \texttt{event9}).

The MDP $\mathcal M$ generated from the environment is shown in Fig.~\ref{fig:MDPexample}.  For this example, we have arbitrarily chosen values for the probability transition function $P_{m}$.  In practice, probabilities of transition under actuation and measurement errors can be obtained from experiments or accurate simulations (see \cite{LaWaAnBe-ICRA10}).  The number of states in the MDP $\mathcal M$ is $|S|=15$.

\begin{figure}[h]
\begin{center}
\includegraphics[scale=0.4]{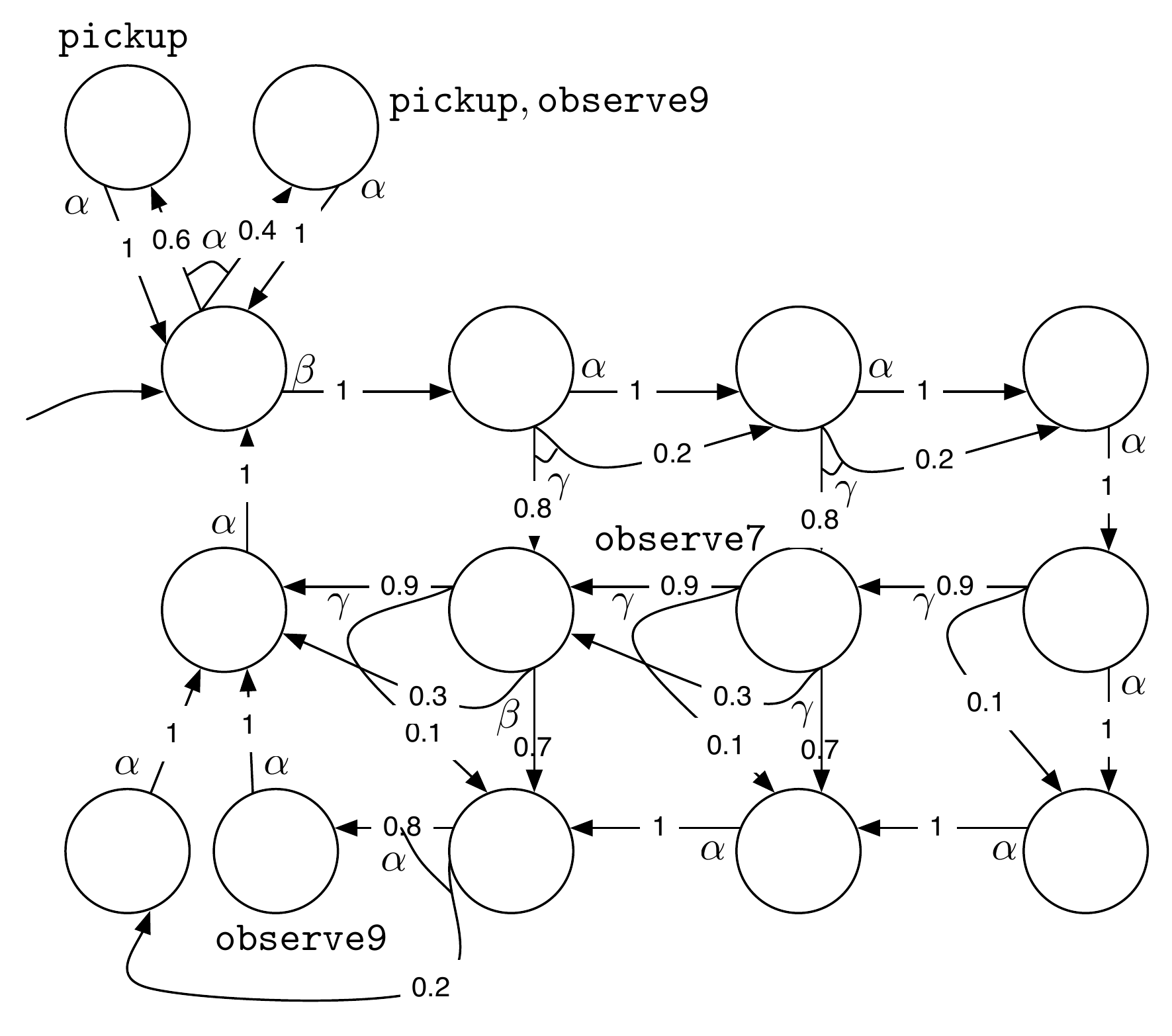}
\caption{MDP $\mathcal M$ generated from the environment with given $U$, $A$, $P_{o}$ and $P_{m}$. The initial state $s_{0}$ is marked by an incoming arrow ($\iota(s_{0})=1$). }%
\label{fig:MDPexample}
\end{center}
\end{figure}

We generated the deterministic Rabin automaton $\mathcal R_{\phi}$ using the ltl2dstar tool (see \cite{ltl2dstar}).  The number of states $|Q|$ is $52$.  Thus, the product MDP $\mathcal M_{\mathcal P}$ has $780$ states.  For the DRA generated, there is only one set in $F$, \ie $F=\{(L,K)\}$, with $1$ state in $L$ and $18$ states in $K$.  Thus, the number of states in $L^{\mathcal P}$ is $15$ and the number of states in $K^{\mathcal P}$ is $270$.  There is one accepting maximum end component in $\mathcal M_{\mathcal P}$, and it contains $17$ states.

Using the implementation of Alg.~\ref{alg:optmaxprob} we computed the maximum probability of satisfying the specification from the initial state and the optimal control strategy.  The Algorithm ran in approximately 7 seconds on a MacBook Pro computer with a 2.5 GHz dual core processor.  For this example the maximum probability is $1$, implying that the corresponding optimal control strategy almost surely satisfies $\phi$.  To illustrate the control strategy, a sample execution is shown in Fig.~\ref{fig:samplePaths}.  

\begin{figure}[h]
\begin{center}
\includegraphics[scale=0.4]{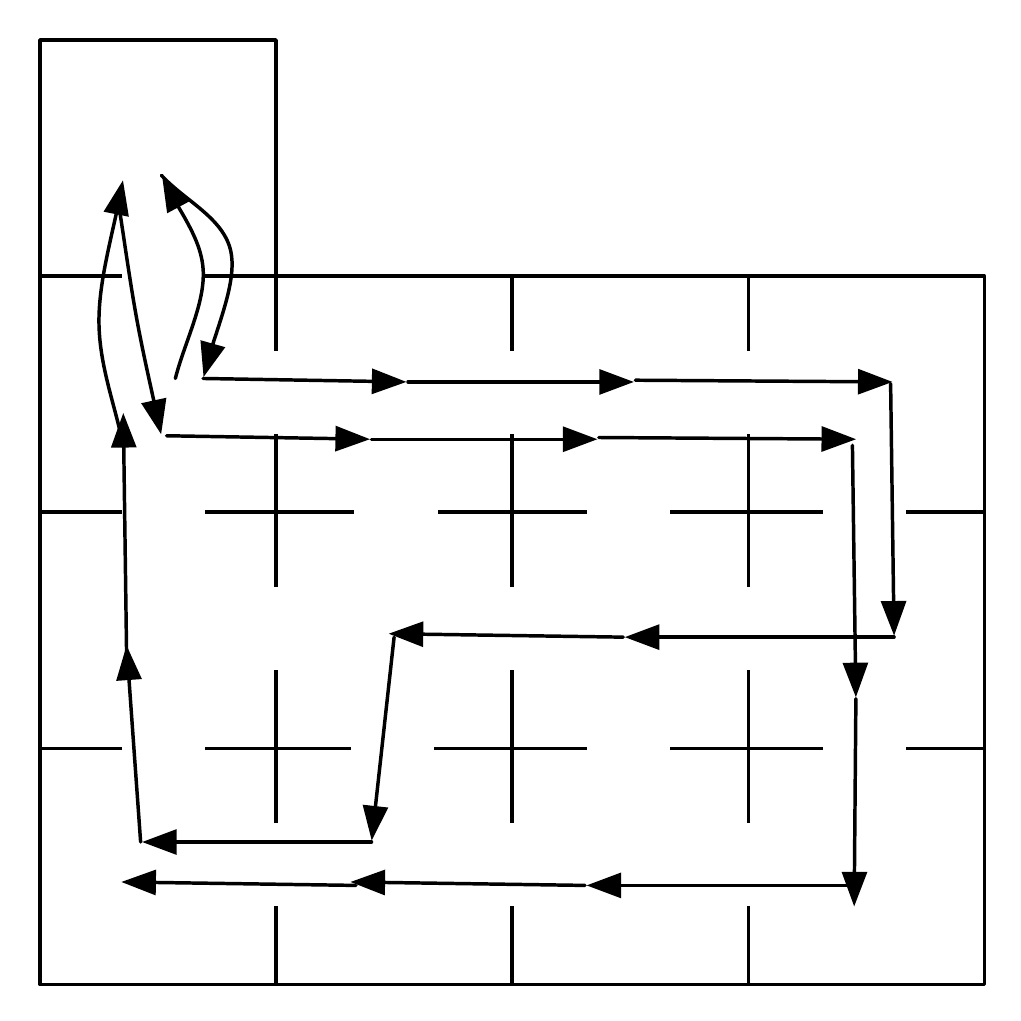}
\caption{A sample path of the robot with the optimal control strategy.  The word observed by the sample path is $\texttt{pickup},\texttt{event7}, \texttt{event9}, \{\texttt{pickup},\texttt{observe9}\}$, $\texttt{event9}, \ldots$.} 
\label{fig:samplePaths}
\end{center}
\end{figure}

\section{Conclusions and Final Remarks}
\label{sec:concl}

We presented a method to generate a robot control strategy that maximizes the probability to accomplish a task.  The robot motion in the environment was modeled as a graph and the task was given as a Linear Temporal Logic (LTL) formula over a set of properties that can be satisfied at the vertices with some probability.  We allowed for noisy sensors and actuators by assuming that a control action enables several transitions with known probabilities.  We reduced this problem to one of generating a control policy for a Markov Decision Process such that the probability of satisfying an LTL formula over its states is maximized.  We then provided a complete solution to this problem adapting existing probabilistic model checking tools.

We are currently pursuing several future directions.  We are looking at proposition observation models that are not independently distributed.  These models arise when the current truth value of the proposition gives information about the future truth value.  We are also looking at methods for optimizing the robot control strategy for a suitable cost function when costs are assigned to actions of an MDP.  The second direction will build on our recent results on optimal motion planning with LTL constraints~\cite{SLS-JT-CB-DR:10b}.


\begin{thebibliography}{10}
\providecommand{\url}[1]{#1}
\csname url@samestyle\endcsname
\providecommand{\newblock}{\relax}
\providecommand{\bibinfo}[2]{#2}
\providecommand{\BIBentrySTDinterwordspacing}{\spaceskip=0pt\relax}
\providecommand{\BIBentryALTinterwordstretchfactor}{4}
\providecommand{\BIBentryALTinterwordspacing}{\spaceskip=\fontdimen2\font plus
\BIBentryALTinterwordstretchfactor\fontdimen3\font minus
  \fontdimen4\font\relax}
\providecommand{\BIBforeignlanguage}[2]{{%
\expandafter\ifx\csname l@#1\endcsname\relax
\typeout{** WARNING: IEEEtran.bst: No hyphenation pattern has been}%
\typeout{** loaded for the language `#1'. Using the pattern for}%
\typeout{** the default language instead.}%
\else
\language=\csname l@#1\endcsname
\fi
#2}}
\providecommand{\BIBdecl}{\relax}
\BIBdecl

\bibitem{Hadas-ICRA07}
H.~Kress-Gazit, G.~Fainekos, and G.~J. Pappas, ``Where's {W}aldo?
  {S}ensor-based temporal logic motion planning,'' in \emph{{IEEE} Int. Conf.
  on Robotics and Automation}, Rome, Italy, 2007, pp. 3116--3121.

\bibitem{Karaman_mu_09}
S.~Karaman and E.~Frazzoli, ``Sampling-based motion planning with deterministic
  $\mu$-calculus specifications,'' in \emph{{IEEE} Conf. on Decision and
  Control}, Shanghai, China, 2009, pp. 2222 -- 2229.

\bibitem{KB-TAC08-LTLCon}
M.~Kloetzer and C.~Belta, ``A fully automated framework for control of linear
  systems from temporal logic specifications,'' \emph{IEEE Transactions on
  Automatic Control}, vol.~53, no.~1, pp. 287--297, 2008.

\bibitem{Loizou04}
S.~G. Loizou and K.~J. Kyriakopoulos, ``Automatic synthesis of multiagent
  motion tasks based on {LTL} specifications,'' in \emph{{IEEE} Conf. on
  Decision and Control}, Paradise Island, Bahamas, 2004, pp. 153--158.

\bibitem{Quottrup04}
M.~M. Quottrup, T.~Bak, and R.~Izadi-Zamanabadi, ``Multi-robot motion planning:
  A timed automata approach,'' in \emph{{IEEE} Int. Conf. on Robotics and
  Automation}, New Orleans, LA, Apr. 2004, pp. 4417--4422.

\bibitem{Tok-Ufuk-Murray-CDC09}
T.~Wongpiromsarn, U.~Topcu, and R.~M. Murray, ``Receding horizon temporal logic
  planning for dynamical systems,'' in \emph{{IEEE} Conf. on Decision and
  Control}, Shanghai, China, 2009, pp. 5997--6004.

\bibitem{Clarke99}
E.~M. Clarke, D.~Peled, and O.~Grumberg, \emph{Model checking}.\hskip 1em plus
  0.5em minus 0.4em\relax MIT Press, 1999.

\bibitem{Emerson90}
E.~A. Emerson, ``Temporal and modal logic,'' in \emph{Handbook of Theoretical
  Computer Science: Formal Models and Semantics}, J.~van Leeuwen, Ed.\hskip 1em
  plus 0.5em minus 0.4em\relax Elsevier, 1990, vol.~B, pp. 995--1072.

\bibitem{Piterman-2006}
N.~Piterman, A.~Pnueli, and Y.~Saar, ``Synthesis of reactive(1) designs,'' in
  \emph{International Conference on Verification, Model Checking, and Abstract
  Interpretation}, Charleston, SC, 2006, pp. 364--380.

\bibitem{Alur00}
R.~Alur, T.~A. Henzinger, G.~Lafferriere, and G.~J. Pappas, ``Discrete
  abstractions of hybrid systems,'' \emph{Proceedings of the IEEE}, vol.~88,
  pp. 971--984, 2000.

\bibitem{Pappas03}
G.~J. Pappas, ``Bisimilar linear systems,'' \emph{Automatica}, vol.~39, no.~12,
  pp. 2035--2047, 2003.

\bibitem{Thomas02}
W.~Thomas, ``Infinite games and verification,'' in \emph{Computer Aided
  Verification}, ser. Lecture Notes in Computer Science, E.~Brinksma and
  K.~Larsen, Eds.\hskip 1em plus 0.5em minus 0.4em\relax Springer, 2002, vol.
  2404, pp. 58--65.

\bibitem{Alfaro95modelchecking}
A.~Dianco and L.~D. Alfaro, ``Model checking of probabilistic and
  nondeterministic systems,'' in \emph{Foundations of Software Technology and
  Theoretical Computer Science}, ser. Lecture Notes in Computer Science.\hskip
  1em plus 0.5em minus 0.4em\relax Springer, 1995, vol. 1026, pp. 499--513.

\bibitem{KNP04b}
M.~Kwiatkowska, G.~Norman, and D.~Parker, ``Probabilistic symbolic model
  checking with {PRISM}: A hybrid approach,'' \emph{International Journal on
  Software Tools for Technology Transfer}, vol.~6, no.~2, pp. 128--142, 2004.

\bibitem{LaWaAnBe-ICRA10}
M.~Lahijanian, J.~Wasniewski, S.~B. Andersson, and C.~Belta, ``Motion planning
  and control from temporal logic specifications with probabilistic
  satisfaction guarantees,'' in \emph{{IEEE} Int. Conf. on Robotics and
  Automation}, Anchorage, AK, 2010, pp. 3227 -- 3232.

\bibitem{baier2004controller}
C.~Baier, M.~Gr{\"o}{\ss}er, M.~Leucker, B.~Bollig, and F.~Ciesinski,
  ``Controller synthesis for probabilistic systems,'' in \emph{Proceedings of
  IFIP TCS}, 2004.

\bibitem{PineauT:AAAI:2002}
J.~Pineau and S.~Thrun, ``High-level robot behavior control using {POMDP}s,''
  in \emph{AAAI Workshop notes}, Menlo Park, CA, 2002.

\bibitem{ZhangZ:JAIR:2001}
N.~L. Zhang and W.~Zhang, ``Speeding up the convergence of value iteration in
  partially observable {M}arkov decision processes,'' \emph{Journal of
  Artificial Intelligence Research}, vol.~14, pp. 29--51, 2001.

\bibitem{baier2008principles}
C.~Baier and J.-P. Katoen, \emph{Principles of Model Checking}.\hskip 1em plus
  0.5em minus 0.4em\relax MIT Press, 2008.

\bibitem{vardi1999probabilistic}
M.~Vardi, ``Probabilistic linear-time model checking: An overview of the
  automata-theoretic approach,'' \emph{Formal Methods for Real-Time and
  Probabilistic Systems}, pp. 265--276, 1999.

\bibitem{gradel2002automata}
E.~Gradel, W.~Thomas, and T.~Wilke, \emph{Automata, logics, and infinite games:
  {A} guide to current research}, ser. Lecture Notes in Computer Science.\hskip
  1em plus 0.5em minus 0.4em\relax Springer, 2002, vol. 2500.

\bibitem{klein2006experiments}
J.~Klein and C.~Baier, ``Experiments with deterministic $\omega$-automata for
  formulas of linear temporal logic,'' \emph{Theoretical Computer Science},
  vol. 363, no.~2, pp. 182--195, 2006.

\bibitem{ltl2dstar}
J.~Klein, ``ltl2dstar - {LTL} to deterministic {S}treett and {R}abin
  automata,'' \url{http://www.ltl2dstar.de/}, 2007.

\bibitem{HS04}
L.~Habets and J.~van Schuppen, ``A control problem for affine dynamical systems
  on a full-dimensional polytope,'' \emph{Automatica}, vol.~40, no.~1, pp.
  21--35, 2004.

\bibitem{Belta-TRO05}
C.~Belta, V.~Isler, and G.~J. Pappas, ``Discrete abstractions for robot
  planning and control in polygonal environments,'' \emph{IEEE Transactions on
  Robotics}, vol.~21, no.~5, pp. 864--874, 2005.

\bibitem{courcoubetis1998markov}
C.~Courcoubetis and M.~Yannakakis, ``Markov decision processes and regular
  events,'' \emph{IEEE Transactions on Automatic Control}, vol.~43, no.~10, pp.
  1399--1418, 1998.

\bibitem{SLS-JT-CB-DR:10b}
S.~L. Smith, J.~T\r{u}mov\'{a}, C.~Belta, and D.~Rus, ``Optimal path planning
  under temporal constraints,'' in \emph{IEEE/RSJ Int. Conf. on Intelligent
  Robots \& Systems}, Taipei, Taiwan, Oct. 2010, to appear.

\end{thebibliography}
\end{document}